\documentclass[pdflatex,sn-mathphys-num]{sn-jnl}
\usepackage{natbib}

\usepackage{graphicx}%
\usepackage{multirow}%
\usepackage{amsmath,amssymb,amsfonts}%
\usepackage{amsthm}%
\usepackage{mathrsfs}%
\usepackage[title]{appendix}%
\usepackage{xcolor}%
\usepackage{textcomp}%
\usepackage{manyfoot}%
\usepackage{booktabs}%
\usepackage{algorithm}%
\usepackage{algorithmicx}%
\usepackage{algpseudocode}%
\usepackage{listings}%
\usepackage{fullpage}
\usepackage{soul}

\theoremstyle{thmstyleone}%
\newtheorem{theorem}{Theorem}
\newtheorem{proposition}[theorem]{Proposition}

\newtheorem{corollary}[theorem]{Corollary}
\newtheorem{definition}[theorem]{Definition}
\newtheorem{remark}[theorem]{Remark}

\newtheorem{claim}[theorem]{Claim}

\newcommand{\op}[1]{\operatorname{#1}}

\newcommand{\bcd}{\op{Bcd}}

\newcommand{\dGH}{\operatorname{d_{GH}}}

\newcommand{\vA}{\textbf{A}}
\newcommand{\vT}{\mathcal{T}}
\newcommand{\vRips}{\mathcal{R}}
\newcommand{\vx}{\textbf{x}}
\newcommand{\vy}{\textbf{y}}
\newcommand{\vz}{\textbf{z}}
\newcommand{\vR}{\mathbb{R}}
\newcommand{\vN}{\mathbb{N}}
\newcommand{\betaexp}{\beta_\text{exp}}
\newcommand{\di}{\operatorname{d_I}}
\newcommand{\dis}{\operatorname{dis}}
\newcommand{\norm}[1]{\left\lVert#1\right\rVert}

\usepackage[normalem]{ulem}

\definecolor{darkgrn}{rgb}{0, 0.75, 0}

\begin{document}

\title[Article Title]{A Stable Measure of Chaos in Dynamical Systems using Persistent Homology}
\author[1]{\fnm{Bala} \sur{Krishnamoorthy}}\email{kbala@wsu.edu}
\author*[1]{\fnm{Elizabeth} \sur{Thompson}}\email{elizabeth.thompson1@wsu.edu}
\affil{
	\orgdiv{Department of Mathematics and Statistics},
	\orgname{Washington State University},
	\orgaddress{
		\street{14204 NE Salmon Creek Ave},
		\city{Vancouver},
		\state{Washington},
		\postcode{98686},
		\country{United States}
	}
}

\abstract{
Many real-world dynamics exhibit chaos, a phenomenon in which neighboring trajectories in the state space of a dynamical system diverge exponentially over time. 
A common measure used for quantifying the degree of this divergence is the maximal Lyapunov exponent, which relies on pairwise Euclidean distances between the trajectories at each time. 
The main limitation of the maximal Lyapunov exponent in practice is its sensitivity to small perturbations in system trajectories.
Since many real-world dynamics are likely to contain some degree of inherent noise, we are motivated to construct a chaos measure that is robust to small trajectory perturbations.
Persistent homology, the study of holes that appear in different dimensions as the points of a data set are thickened over time, has guaranteed theoretical stability under such added noise.
As such, we propose a novel, 0-dimensional persistent homology based measure of chaos termed the 0-persistence exponent and prove its theoretical stability.
We show that if a system is chaotic, then the 0-persistence exponent is non-negative by proving that positive Lyapunov measures imply non-negative 0-persistence measures, and further discuss when strict positivity of 0-persistence measures occur.
Additionally, we prove the existence of an upper bound on our measure, and show its greater experimental stability on the Lorenz and R\"{o}ssler systems describing fluid convection and taffy pulling.
We present an algorithm for computing the 0-persistence exponent given a single univariate time series with $N$ points from a dynamical system that runs in $O(N^2 \log N)$ time.
We finally show the greater experimental stability of the 0-persistence exponent on time series data depicting a Belousov-Zhabotinsky chemical reaction, which transitions from periodicity to chaos and back as the system evolves in time.
We present experimental results which verify that positive Lyapunov exponents imply positive 0-persistence exponents under sufficient conditions through high correlation between both measures on the Lorenz and R\"{o}ssler systems.}


\maketitle

\section{Introduction: Dynamical Systems and Chaos}
\label{sec:dynamical_systems}

Many real-world systems can be characterized by differential equations describing the dynamics of a collection of parameters over time. 
As such, they are commonly referred to as dynamical systems, which model a variety of phenomenon including but not limited to fluid dynamics, population growth, and chemical reactions. 
Dynamical systems are typically studied via their trajectories in state space, which, given an initial condition of the system, are obtained by integrating its differentials to yield a collection of time series and plotting these series together as multidimensional points in Euclidean space. 
Another common approach is to apply a sliding windows embedding to a single time series to produce a state space trajectory for the system \cite{perea2015,KANTZ199477}.
A system feature of particular interest are control parameters that drastically change the system's behavior. 
Commonly studied behaviors are periodicity and chaoticity, and the study of which control parameters influence a system to behave in either of these fashions, and the extent of chaos when present. 
Intuitively, a system is periodic if its state space exhibits ``loop-like'' structure, in which as one traverses a point cloud trajectory, locations are revisited at a given frequency. 
On the other hand, a system is considered chaotic if as one traverses two neighboring trajectories, they become exponentially farther apart over time, a phenomenon referred to as the exponential divergence of nearby trajectories \cite{hilborn2006}.
Several methods have been used to study periodicity and chaos, such as quantifying (\cite{hilborn2006,RUDISULI2013813}) or detecting (\cite{10.1063/5.0102421,10.1063/1.4983840,SHAH2025116054,TEMPELMAN2020132446}) these behaviors, as well as approximating the dimension of the geometric space occupied by a state-space trajectory (\cite{JaSc2020,GRASSBERGER1983189}).
We provide an example in \autoref{fig:ex_lorenz_system} of a well known dynamical system, the Lorenz attractor (see \autoref{ssec:Lorenz_results} for details). 
This system models fluid convection and is defined by three differentials, which when integrated produce three time series in state space ($\mathbb{R}^3$). 
We show a chaotic system on the top and a periodic system on the bottom. 

\begin{figure}[ht!]
\centering
    \includegraphics[width=0.9\textwidth]{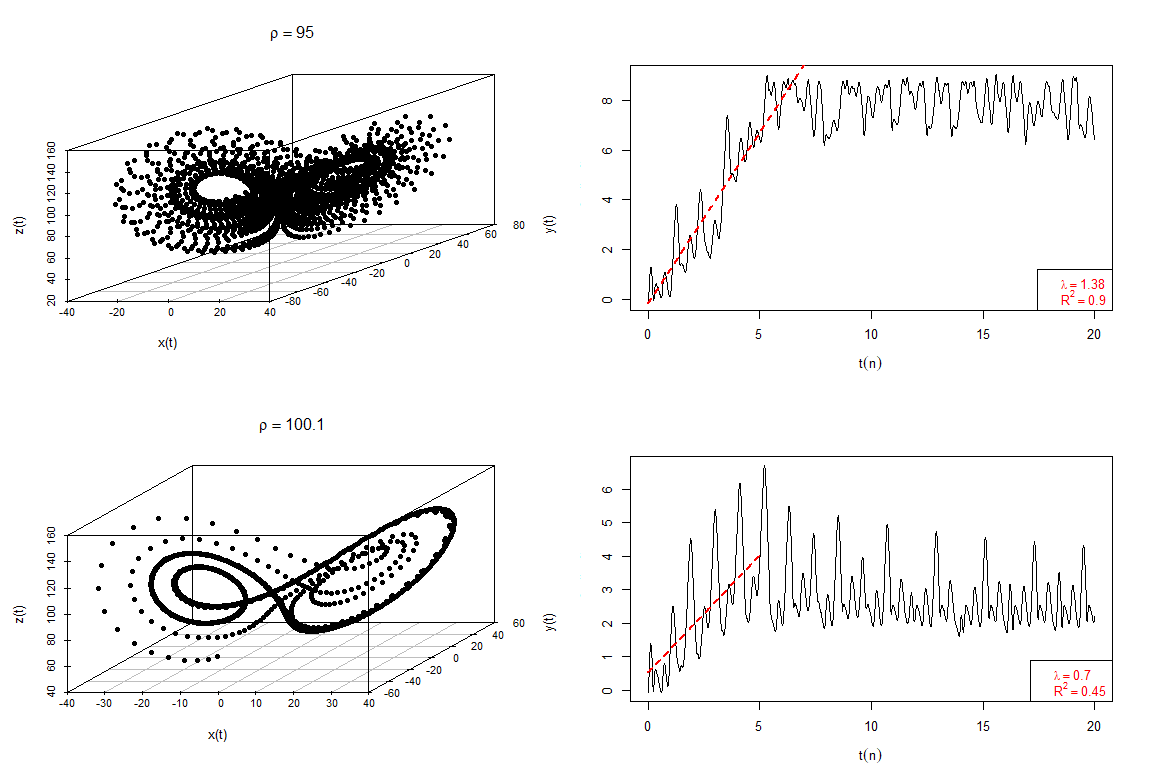}
    \caption{A chaotic (top left) and periodic (bottom left) Lorenz system, as well as their respective maximal Lyapunov exponent measurements (top right and bottom right).
    $\rho$ is the control parameter varied to detect changes in system behavior (see \autoref{ssec:Lorenz_results} for details).}
    \label{fig:ex_lorenz_system}
\end{figure}

\subsection{Limitations of Measuring Chaos}
\label{sec:limitations}

A common method for quantifying the chaos in a dynamical system given a pair of its neighboring trajectories (i.e., those whose initial conditions are close) is via the maximal Lyapunov exponent.
This exponent measures the largest rate of exponential divergence of two neighboring trajectories over time in state space, and in practice is estimated by calculating the slope of a linear regression on the logarithm of pairwise Euclidean distances (divided by the distance between initial conditions) against time.
Theoretically, positive maximal exponents indicate chaos and zero or negative values indicate periodicity.
In practice, higher positive maximal exponents indicate greater chaos, whereas lower positive, zero, or negative values indicate greater periodicity \cite{hilborn2006,RUDISULI2013813}. 
In the right column of \autoref{fig:ex_lorenz_system},
we show an example of the maximal Lyapunov exponent for a chaotic and a periodic Lorenz system.
We can see that when the system is chaotic, the estimated exponent is larger (1.38 vs 0.7).

\begin{figure}[ht!]
\centering
\includegraphics[width=0.45\textwidth]{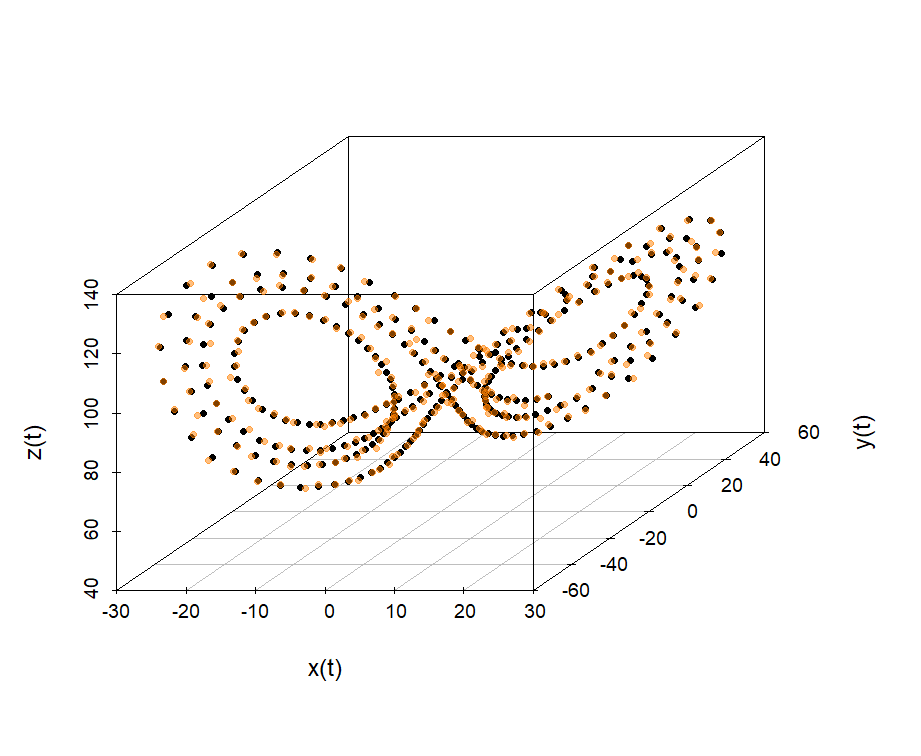}
\includegraphics[width=0.45\textwidth]{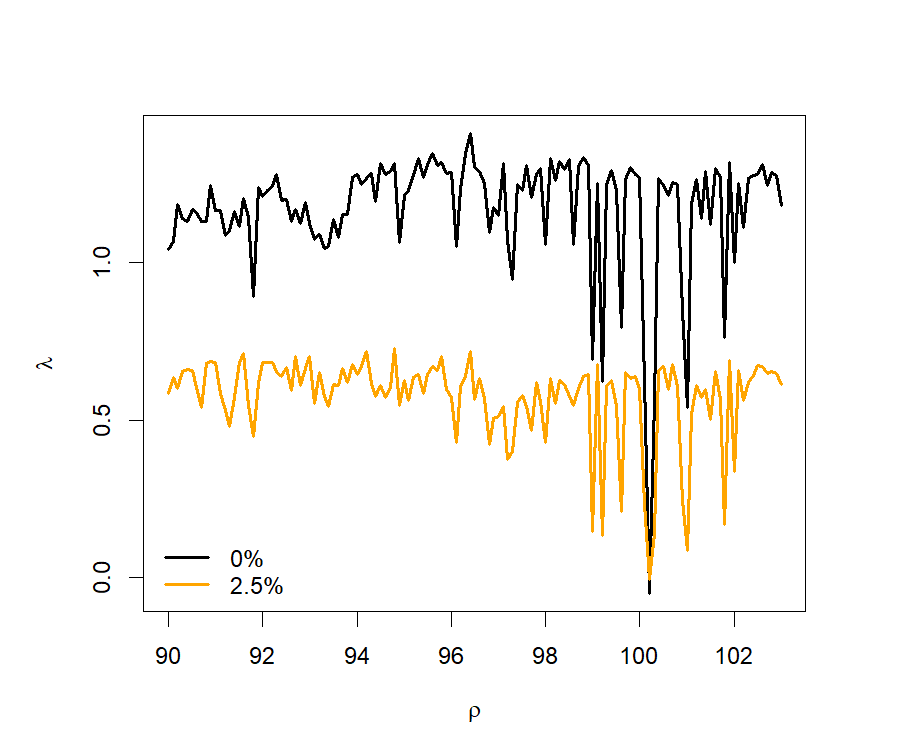}
\caption{One sample of two neighboring chaotic Lorenz trajectories with initial conditions $(-13,-14,47)$ (black) and $(-12.981,-14,007,46.999)$ (orange) with 2.5\% added Gaussian noise when $\rho = 95$ (left), and the average maximal Lyapunov exponents over 20 samples of neighboring trajectories with added Gaussian noise for increasing values of $\rho$ (right).
There is a marked decrease in the exponent even for the small amount of noise added.}
\label{fig:instability_Lyapunov_Lorenz}
\end{figure}

One of the main practical limitations of the maximal Lyapunov exponent is its lack of robustness to small perturbations in the system trajectories generated by initial conditions.
We suspect this sensitivity is due to its reliance on Euclidean distance alone for computation.
We show an example in \autoref{fig:instability_Lyapunov_Lorenz} of its instability when measuring chaos in the Lorenz system with added noise. 
We plot the average maximal Lyapunov exponent over 20 pairs of neighboring trajectories with added Gaussian noise.
To generate each sample of the first noisy trajectory, we define its initial condition as $ \vx(t_0) = (-13,-14,47) + \mathcal{N}(0,0.01)$, where we denote the Gaussian distribution by $\mathcal{N}$.
We next add Gaussian noise at level $\sigma$ to the trajectory $\{\vx(t)\} \in \vR^3$ generated by $\vx(t_0)$, where we define its  perturbation in the $d$th direction as $\vx^d_\sigma(t) = \vx_d(t) + \sigma \cdot \gamma_d$, where $\gamma_d$ is the standard deviation of $\{\vx_d(t)\}$, the $d$th dimension of $\{\vx(t)\}$.
We define each sample of the neighboring trajectory of $\vx(t_0)$ by first generating its initial condition as $\vy(t_0) = \vx(t_0) + \mathcal{N}(0,0.01)$, and then adding noise to the trajectory $\{\vy(t)\}$ that it generates similarly. 
For each sample of neighboring trajectories, we estimate $\lambda$ by computing the slope of a linear regression on the $\ln \left( \frac{||\vx_\sigma(t) - \vy_\sigma(t)||}{||\vx_\sigma(t_0) - \vy_\sigma(t_0)||} \right) $ vs $t$ plot for $t = [0,7]$, where $\vx_\sigma(t) = \left(\vx_\sigma^1(t),\vx_\sigma^2(t),\vx_\sigma^3(t)\right)$ and $\vy_\sigma(t) = \left(\vy_\sigma^1(t),\vy_\sigma^2(t),\vy_\sigma^3(t)\right)$.

Notice that when we add just 2.5\% noise (i.e., $\sigma = 0.025$) to a pair of neighboring trajectories, the average maximal Lyapunov exponent drops substantially. 
Since dynamical systems model real-world phenomena, we expect there to be some inherent trajectory noise.
On the other hand, persistent homology studies holes that appear as one thickens the points in a metric space over time, and has proven theoretical stability to small data perturbations \cite{ChdeSOu2014,perea2015}.
As such, we construct a persistence-based measure of chaos that is stable to such added noise, which we term the \textit{0-persistence exponent}.
The intuition behind this measure is that neighboring trajectories from more periodic systems will be closer together and hence their points will merge much more quickly during Vietoris-Rips filtration, while those from chaotic systems (i.e., those with exponential divergence) will contain points that take longer to merge together.
Our measure captures the rate of this merging over time. 
See \autoref{sec:persistence_exponent} for more details.

\subsection{Our Contributions and Organization}
\label{sec:contributions}

We define a novel measure of chaos in dynamical systems based on 0-dimensional persistent homology termed the \textit{0-persistence exponent} (Definition \ref{def:Persistence_exponent}).
Unlike the commonly used maximal Lyapunov exponent (Definition \ref{def:lyapunov_exponent}), the 0-persisence exponent is stable to small perturbations in state-space trajectories.
We prove the theoretical stability of the 0-persistence exponent (\autoref{thm:stability_Persistence_exp}) as well as show its greater experimental stability on the Lorenz (\autoref{fig:Lorenz_results}) and the R\"{o}ssler (\autoref{fig:Rossler_results}) systems.
We also derive theoretical upper and lower bounds on the 0-persistence exponent (Remark \ref{rmk:bounds_Persistence_exp}), which allow us to deduce the maximum possible degree of exponential divergence of two neighboring trajectories for any range of filtration parameters.
We present an algorithm to compute the 0-persistence exponent from a single time series embedding (Algorithm \ref{alg:Persistence}) and deduce its computational complexity (Remark \ref{rmk:complexity_persistence_algo}).
We finally use our algorithm to show greater experimental stability of the 0-persistence exponent when compared to the maximal Lyapunov exponent on real time series data describing the change in concentration of a catalyst in an autocatalytic chemical reaction over time (\autoref{fig:Real_results}).

We introduce definitions related to dynamical systems (\autoref{ssec:defs_systems_chaos}), persistent homology (\autoref{ssec:PH_dist_metrics}), and our 0-persistence exponent (\autoref{ssec:persistence_exp}).
We then outline previous work done in both detecting and quantifying periodicity and chaos in dynamical systems, followed by a brief review of the theoretical stability of persistent homology (\autoref{sec:related_work}).
In \autoref{sec:persistence_exponent}, we introduce our 0-persistence exponent and the intuition behind its construction, as well as a proof of its stability and its bounds (\autoref{thm:stability_Persistence_exp} and Proposition \ref{rmk:bounds_Persistence_exp}). 
We show that if a system is chaotic then our 0-persistence exponent is guaranteed to be non-negative (\autoref{thm:lambda_positive_implies_betaexp_nonnegative}), and strictly positive under sufficient conditions (\autoref{cor:positive_lambda_implies_positive_betaexp} and Remark \ref{rmk:when_lambda_positive_implies_betaexp_positive}).
In \autoref{sec:algo}, we present an algorithm for computing the 0-persistence exponent on real time series data and discuss its computational complexity.
We present experimental results that highlight the greater stability of our 0-persistence exponent on various dynamical systems with known behavior (\autoref{ssec:Lorenz_results} and \autoref{ssec:Rossler_results}), as well as on real time series data (\autoref{ssec:Real_results}).
We finally show experimentally that the Lyapunov and 0-persistence exponents are highly correlated for the Lorenz and R\"{o}ssler systems (\autoref{ssec:correlation}), which suggests that positive Lyapunov measures yield positive 0-persistence measures for sufficiently many trajectory points.
We end with discussion, limitations, and future work in \autoref{sec:discussion_future}.

\subsection{Related Work}
\label{sec:related_work}

One common method used to detect chaos in a dynamical system (Definition \ref{def:chaotic_system}) is through a bifurcation diagram \cite{hilborn2006}.
This is typically a plot of the local extrema of one time series defining a dynamical system against varying values of a control parameter.
These diagrams are used to determine the control parameters for which a system is chaotic or periodic.
A widely used measure for quantifying chaos in a system $X$ given a pair of its neighboring trajectories (Definition \ref{def:neighboring_trajectories}) in state space is the maximal Lyapunov exponent (Definition \ref{def:lyapunov_exponent}), which we denote by $\lambda$.
The maximal Lyapunov exponent measures the rate of exponential divergence (in Euclidean distance) of two neighboring trajectories over time in state space \cite{hilborn2006,RUDISULI2013813}.
Positive values of $\lambda$ typically indicate chaos, whereas zero or negative values of $\lambda$ indicate periodicity.
Persistent homology has also been used to detect system bifurcations using topological summaries called CROCKER plots \cite{10.1063/1.4983840}.
CROCKER plots of dynamical systems are heatmaps of $p$-dimensional Betti numbers for increasing filtration parameters against increasing values of a given control parameter, and Mittal and Gupta \cite{10.1063/1.4983840} experimentally discovered similarities between the behavior of the Lyapunov exponent and that of the Betti vector norm of the columns in the 1-CROCKER plot of a Lorenz system against increasing values of a control parameter.
More recently, persistent homology has been used to predict chaotic behavior in the R\"{o}ssler system \cite{10.1063/5.0268340}, where the authors used Betti curves from Vietoris-Rips filtration to produce stable classifications of dynamics for varying control parameters. 
It is known that trajectories (Definition \ref{def:trajectory}) of chaotic systems must exist in at least three dimensions \cite{hilborn2006}.
These systems can be highly complex and persistent homology has become a popular method for measuring the dimension of the geometric space embedding chaotic trajectories \cite{JaSc2020}.
The main intuition is that more chaotic systems will produce higher growth rates and hence take up larger geometric dimensions in state space. 

Persistent homology has been proven to be stable to small perturbations in input point clouds \cite{ChdeSOu2014,Le2015}.
Specifically, given a pair of metric spaces $(X,d_X)$ and $(Y,d_Y)$ and their corresponding persistence barcodes $\bcd(\vRips(X))$ and $\bcd(\vRips(Y))$ (Definition \ref{def:barcode}), the following relation is guaranteed to hold:
\begin{equation*}
    \di \Big( \bcd \big( \vRips(X) \big), \bcd \big( \vRips(Y) \big) \Big) \leq 2 \dGH(X,Y).
\end{equation*}

That is, small changes in $X$ as measured by the Gromov-Hausdorff distance (Definition \ref{def:gromov_hausdorff_dist}) lead to only small changes in $\bcd(\vRips(X))$ as measured by the interleaving distance (Definition \ref{def:interleaving_dist}).
This stability is the key motivation for our construction of the 0-persistence exponent (Definition \ref{def:Persistence_exponent}) as a stable alternative to the maximal Lyapunov exponent to measure chaos in dynamical systems.

\section{Definitions}
\label{sec:definitions}

\subsection{Dynamical Systems and Chaos}
\label{ssec:defs_systems_chaos}

This section contains an overview of necessary terminology regarding dynamical systems and chaos. 
More information can be found in, e.g., the book by Hilborn \cite{hilborn2006}.

\begin{definition}
  [Dynamical System]
  \label{def:dynamical_system}
  A \textbf{dynamical system} $X$ is a system whose dynamics can be modeled by a collection of $p$ differential equations $\left\{ \frac{dX_1(\vA,t)}{dt}, \dots, \frac{dX_p(\vA,t)}{dt} \right\}$,
  each as a function of time $t \in \vT$ and $m$ control parameters $\vA = \{ A_1, \dots, A_m\}$.
\end{definition}

\begin{definition}
  [State Space]
  \label{def:state_space}
  The \textbf{state space} of a dynamical system $X$ is given by $\Big( X_1(\vA,t), \dots, X_p(\vA,t) \Big) \in \vR^p$ for $t \in \vT$.
\end{definition}

\begin{definition}
    [Trajectory]
    \label{def:trajectory}
    The \textbf{trajectory} of a dynamical system $X$ in its state space evaluated at times $\vT = \{t_0,\dots,t_{T-1}\}$ for $T \in \vN$ is given by 
    $\vx_\vA(\vT) = \{ \vx_\vA(t_0), \dots, \vx_\vA(t_{T-1}) \}$ where $\vx_\vA(t_n) = \Big( X_1(\vA,t_n), \dots, X_p(\vA,t_n) \Big) \in \vR^p$.
\end{definition}

\begin{definition}
    [Neighboring trajectories]
    \label{def:neighboring_trajectories}
    Given an initial condition $\vx_\vA(t_0)$ of a dynamical system $X$, we define its corresponding trajectory as $\vx_\vA(\vT)$. 
    Another trajectory $\vy_\vA(\vT)$ is a \textbf{neighboring trajectory} of $\vx_\vA(\vT)$ if $\norm{\vx_\vA(t_0) - \vy_\vA(t_0)} \leq \epsilon$ for $\epsilon > 0$.
\end{definition}

\begin{definition}
    [Maximal Lyapunov Exponent]
    \label{def:lyapunov_exponent}
    Given a pair of neighboring trajectories $\vx_\vA(\vT)$ and $\vy_\vA(\vT)$ in the state space of a dynamical system $X$, the \textbf{maximal Lyapunov exponent} of $X$ is given by 
    
    \[ \lambda = \lim_{t_n \rightarrow \infty} \lim\limits_{d_0 \rightarrow 0} \frac{1}{t_n} \ln \left( \frac{d_n}{d_0} \right), \] 
    
    \noindent where $d_n = \norm{\vx_\vA(t_n) - \vy_\vA(t_n)}$ and $d_0 = \norm{\vx_\vA(t_0) - \vy_\vA(t_0)}$. 
    Unless otherwise stated, we will refer to the maximal Lyapunov exponent as the Lyapunov exponent in the remaining sections of this paper.
\end{definition}

\begin{definition}\label{def:chaotic_system}
    [Chaotic System] 
    A dynamical system is chaotic if its maximal Lyapunov exponent is positive.
\end{definition}

\subsection{Persistent Homology and Distance Metrics}
\label{ssec:PH_dist_metrics}

This section provides necessary definitions of persistent homology and distance metrics for topological summaries. 
More information can be found in relevant publications \cite{ScChLuRaOb2017,ShSh2013}.

\begin{definition}
  [Vietoris-Rips Complex]
  \label{def:Vietoris_rips_cplx}
  Given a point cloud $X \in \vR^n$ and a positive radius $\epsilon > 0$, the \textbf{Vietoris-Rips complex} of $X$ at $\epsilon$ is given by $\vRips_\epsilon(X) = \{ \sigma \subset X  : \norm{\vx_i - \vx_j}  \leq 2\epsilon, \, \forall \, \vx_i, \vx_j \in \sigma \}$.
\end{definition}

\begin{definition}
  [Vietoris-Rips Filtration]
  \label{def:Vietoris_rips_filt}
  Given a point cloud $X \in \vR^n$ and a collection of positive radii $\{ \epsilon_1=0, \dots, \epsilon_R\}$, the \textbf{Vietoris-Rips filtration} of $X$ is given by the nested sequence of subcomplexes $\vRips_{\epsilon_1}(X) = X \subset \vRips_{\epsilon_2}(X) \subset \dots \subset \vRips_{\epsilon_R}(X)$.
\end{definition}

\begin{definition}
  [Persistence Barcode]
  \label{def:barcode}
  The \textbf{persistence barcode} of a point cloud $X \in \vR^n$ is given by $\bcd(\vRips(X)) = \{[b_p^i,d_p^i) : p = 0,1, \dots, \, i = 1, \dots, m_p\}$, where $d_p^i-b_p^i$ is the lifetime of the $i$-th $p$-dimensional class of holes in the Vietoris-Rips filtration on $X$. 
\end{definition}

\begin{definition}
  [Betti Number]
  \label{def:betti_num}
  Given the persistence barcode of a point cloud $X \in \vR^n$, the \textbf{$p$-th Betti number} of $X$ at filtration radius $\epsilon > 0$ is given by $\beta_p^\epsilon(X) = \sum_{i=1}^{m_p} \delta_\epsilon \big([b_p^i,d_p^i) \big)$, where $\delta_\epsilon \big( [b_p^i,d_p^i) \big) = 1$ if $b_p^i \leq \epsilon < d_p^i$ and is zero otherwise.
\end{definition}

\begin{definition}
  [Correspondence]
  \label{def:correspondence}
  A \textbf{correspondence} $f$ between two metric spaces $(X,d_X)$ and $(Y,d_Y)$ is a relation on $X \times Y$ such that for all $\vx, \vx' \in X$, there exists some $\vy, \vy' \in Y$ such that $f(\vx)=\vx'$ and $f(\vy)=\vy'$.
\end{definition}

\begin{definition}
  [Distortion] 
  \label{def:distortion}
  Given two metric spaces $(X,d_X)$ and $(Y,d_Y)$, the distortion of a correspondence $f$ between them is given by 
  $\dis(f) = \sup\limits_{\vx,\vx'} \Big\{ \Big\vert \norm{\vx-\vx'} - \norm{f(\vx)-f(\vx')} \Big\vert \Big\}$ for $\vx,\vx' \in X$.
\end{definition} 

\begin{definition}
  [Gromov-Hausdorff Distance]
  \label{def:gromov_hausdorff_dist}
  The \textbf{Gromov-Hausdorff distance} between two metric spaces $(X,d_X)$ and $(Y,d_Y)$ is given by $\dGH(X,Y) = \frac{1}{2} \inf \{ \dis(f) \}$ among all possible correspondences $f$ on $X \times Y$. 
\end{definition}

\begin{definition}
  [Interleaving Distance]
  \label{def:interleaving_dist}
  The \textbf{interleaving distance} between two non-increasing piecewise constant functions $g,h :[0,\infty) \rightarrow [0,\infty)$ is given by $\di(g,h) = \inf \{ \alpha > 0 : g(v) \geq h(v + \alpha), \text{ } h(v) \geq g(v+\alpha) \}$ for all $v \in [0,\infty)$.
\end{definition}

\subsection{The 0-Persistence Exponent}
\label{ssec:persistence_exp}

We now state a special case of the interleaving distance that measures the change between two $0$-dimensional Betti curves.
We then provide the definition of our novel (0-dimensional) persistence-based measure of chaos in dynamical systems.

\begin{definition}
  [Interleaving Distance Between 0-Dimensional Betti Curves]
  \label{def:interleaving_dist_betti_curves}
  Given a pair of neighboring trajectories $Z = \vx_\vA(\vT) \cup \vy_\vA(\vT)$, their perturbation $Z'=\vx'_\vA(\vT) \cup \vy'_\vA(\vT)$, and a partition of filtration radii $0 \leq \epsilon_1 < \dots < \epsilon_R$, we define the 0-dimensional Betti curve of $Z$ and $Z'$ as the non-increasing piecewise constant functions $\beta_0^{\epsilon_r}(Z),\beta_0^{\epsilon_r}(Z') : [0,\epsilon_R] \rightarrow [0,2T]$.
  The \textbf{interleaving distance between the 0-dimensional Betti curves} is given by 
  $\di(\beta_0^{\epsilon_r}(Z),\beta_0^{\epsilon_r}(Z')) = \inf \{ \alpha > 0 : \beta_0^{\epsilon_r}(Z) \geq \beta_0^{\epsilon_r + \alpha}(Z'),\beta_0^{\epsilon_r}(Z') \geq \beta_0^{\epsilon_r + \alpha}(Z) \}$. 
\end{definition}

\begin{definition}
  [0-Persistence Exponent]
  \label{def:Persistence_exponent}
  The \textbf{0-persistence exponent} of a dynamical system $X$ is given by

  \[ \betaexp(X) = - \lim\limits_{\epsilon_r \rightarrow \infty} \frac{1}{\epsilon_r} \ln \left( \frac{\beta_0^{\epsilon_r}(Z)}{\beta_0^{\epsilon_1}(Z)} \right), \] 
  
  \noindent where $\beta_0^{\epsilon_r}(Z)$ is the $r$-th $0$-dimensional Betti number corresponding to the union of neighboring trajectories $Z = \vx_\vA(\vT) \cup \vy_\vA(\vT)$ in the state space of $X$.
\end{definition}

\section{The Persistence Exponent and Stability}
\label{sec:persistence_exponent}

We now introduce a novel, persistence-based measure of chaos termed the 0-persistence exponent, which we denote as $\betaexp$.
Given a pair of neighboring trajectories $\vx_\vA(\vT)$ and $\vy_\vA(\vT)$ in state space of a dynamical system $X$, we analyze the evolution of the connected components in the Vietoris-Rips filtration $\vRips(Z)$ of their union $Z = \vx_\vA(\vT) \cup \vy_\vA(\vT)$ over time.  
The main intuition is that we expect periodic trajectories to contain points that merge much more quickly at first, leaving less connected components to merge later on and hence producing a lower rate of merging long term.
On the other hand, chaotic trajectories contain points that are likely to take exponentially longer to merge as the filtration radius increases, which produce higher rates of merging in the long term.
We capture the rate of merging of connected components of $\vRips(Z)$ via the rate of growth in the negative logarithm of the number of connected components over time, where time indicates the filtration radius. 
Moreover, we hypothesize that chaotic systems with sufficiently large enough degrees of exponential divergence satisfy the relationship
$\beta_0^{\epsilon_r}(Z) = \beta_0^{\epsilon_1}(Z)e^{-\betaexp \cdot \epsilon_r}$, where the number of connected components decreases exponentially as a function of the filtration radius.
Hence, the rate of exponential divergence is given by 

\begin{equation*}
    \betaexp(X) = - \frac{1}{\epsilon_r} \ln \left( \frac{\beta_0^{\epsilon_r}(Z)}{\beta_0^{\epsilon_1}(Z)} \right).
\end{equation*}

Using stability properties of persistent homology, we prove that small changes in a pair of neighboring trajectories of a dynamical system guarantee small changes in the $0$-dimensional Betti curve of their union. 
Hence we prove stability of the 0-persistence exponent to small perturbations in the neighboring trajectories of dynamical system's state space. 
We finally deduce that under certain conditions, if a system is chaotic (i.e. $\lambda > 0$), then the 0-persistence exponent will be positive (i.e. $\betaexp > 0$) (Remark \ref{rmk:when_lambda_positive_implies_betaexp_positive}).
We include experimental results that support this claim in the following section, as well as prove that positive $\lambda$ guarantees non-negative $\betaexp$ in \autoref{thm:lambda_positive_implies_betaexp_nonnegative} of this section. 

\begin{theorem}
  [Stability of the 0-Persistence Exponent]
  \label{thm:stability_Persistence_exp}
  Let $(Z = \vx_\vA(\vT) \cup \vy_\vA(\vT),\norm{\cdot})$ and $(Z' = \vx'_\vA(\vT) \cup \vy'_\vA(\vT),\norm{\cdot})$ be two pairs of neighboring trajectories in the state space of a dynamical system $X$ where $\vT = \{t_0,\dots,t_{T-1}\}$.
  Given a partition $0 = \epsilon_1 < \dots < \epsilon_r < \dots <\epsilon_R$ of Vietoris-Rips filtration radii, the following relation holds:
  \begin{align*}
    \di(\beta_0^{\epsilon_r}(Z),\beta_0^{\epsilon_r}(Z')) \leq 2 \dGH(Z,Z'),
  \end{align*}
where $\beta_0^{\epsilon_r}(Z),\beta_0^{\epsilon_r}(Z'):[0,\epsilon_R]\rightarrow [0,2T]$.
\end{theorem}

\begin{proof}
    Suppose $\dGH(Z,Z') = \gamma$. 
    Then for all correspondences $f$ on $Z \times Z'$, we have that $\inf \{ \dis(f) \} = 2 \gamma$.
    Since $Z$ and $Z'$ are both finite, there must exist some correspondence $f^\ast$ such that
    \begin{align*}
        \dis(f^\ast) = \sup_{\vz_i,\vz_j \in Z} \Big\{ \Big\vert \norm{\vz_i-\vz_j} - \norm{f^\ast(\vz_i) - f^\ast(\vz_j)} \Big\vert \Big\} = 2\gamma.
    \end{align*} 
    \noindent Hence for all $\vz_i,\vz_j \in Z$, $\Big\vert \norm{\vz_i-\vz_j} - \norm{f^\ast(\vz_i) - f^\ast(\vz_j)} \Big\vert \leq 2\gamma$, which implies that 
    \begin{align}
        \norm{\vz_i-\vz_j} & \leq \norm{f^\ast(\vz_i)-f^\ast(\vz_j)} + 2\gamma \label{eq:stab1} \\
        \norm{f^\ast(\vz_i)-f^\ast(\vz_j)} & \leq \norm{\vz_i-\vz_j} + 2\gamma\,. \label{eq:stab2} 
    \end{align}

    \begin{claim} \label{clm:stbpf1}
      $\beta_0^{\epsilon_r}(Z) \geq \beta_0^{\epsilon_r+2\gamma}(Z')$.
    \end{claim}

    \begin{proof}
      Let $(\mathbf{a},\mathbf{b})$ be an edge in $\vRips_{\epsilon_r}(Z)$. 
      Then $\norm{\mathbf{a}-\mathbf{b}} \leq \epsilon_r$ by definition.
      Hence by (\autoref{eq:stab2}), $\norm{f^\ast(\mathbf{a})-f^\ast(\mathbf{b})} \leq \norm{\mathbf{a}-\mathbf{b}} + 2 \gamma \leq \epsilon_r + 2 \gamma$.
      Hence $(f^\ast(\mathbf{a}), f^\ast(\mathbf{b}))$ is an edge in $\vRips_{\epsilon_r+2\gamma}(Z')$ and $\vRips_{\epsilon_r}(Z) \subseteq \vRips_{\epsilon_r+2\gamma}(Z')$.
      That is, the Vietoris-Rips complex of $Z'$ at $\epsilon_r + 2\gamma$ is more connected than the Vietoris-Rips complex of $Z$ at $\epsilon_r$, and hence our claim holds.
    \end{proof}

    \begin{claim} \label{clm:stbpf2}
      $\beta_0^{\epsilon_r}(Z') \geq \beta_0^{\epsilon_r+2\gamma}(Z)$.
    \end{claim}

    \begin{proof}
      Let $(\mathbf{a}', \mathbf{b}')$ be an edge in $\mathcal{R}_{\epsilon_r}(Z')$. 
      Then $\norm{\mathbf{a}'-\mathbf{b}'} \leq \epsilon_r$ by definition. 
      Hence there exists some $\mathbf{a},\mathbf{b} \in Z$ such that $f^\ast(\mathbf{a}) = \mathbf{a}'$ and $f^\ast(\mathbf{b}) = \mathbf{b}'$.
      Hence by  (\autoref{eq:stab1}), $\norm{\mathbf{a}-\mathbf{b}} \leq \norm{\mathbf{a}'-\mathbf{b}'} + 2 \gamma \leq \epsilon_r + 2 \gamma$.
      Therefore, $(\mathbf{a}, \mathbf{b})$ is an edge in $\vRips_{\epsilon_r+2\gamma}(Z)$.
      Then $\vRips_{\epsilon_r}(Z') \subseteq \vRips_{\epsilon_r+2\gamma}(Z)$ and hence the Vietoris-Rips complex of $Z$ at $\epsilon_r + 2\gamma$ is more connected than that of $Z'$ at $\epsilon_r$ and therefore our claim is satisfied.
    \end{proof}

    Claims \ref{clm:stbpf1} and \ref{clm:stbpf2} together imply that $2\gamma$ is one possible candidate for the interleaving distance between the 0-dimensional Betti curves. 
    Since this interleaving distance is the infimum of all possible solutions, we get that
    $\di(\beta_0^{\epsilon_r}(Z),\beta_0^{\epsilon_r}(Z')) \leq 2\gamma = 2\dGH(Z,Z')$.
\end{proof}

\begin{proposition}
  [Bounds on the 0-Persistence Exponent]
  \label{rmk:bounds_Persistence_exp}
  Let $\epsilon_{r_{\min}} < \epsilon_{r_{\max}}$ and assume that $\epsilon_1$ need not be zero.
  Let $\betaexp$ be the 0-persistence exponent of a dynamical system $X$ on a pair of neighboring trajectories $Z = \vx_\vA(\vT) \cup \vy_\vA(\vT)$ for $\vT = \{t_0,\dots,t_{T-1}\}$.
  If $\betaexp$ was obtained via linear regression on $[\epsilon_{r_{\min}},\epsilon_{r_{\max}}]$, then
  
  \[ 0 \leq \betaexp(X) \leq \frac{- \log \left( \frac{1}{2T} \right)}{\epsilon_{r_{\max}} - \epsilon_{r_{\min}}}.\]
\end{proposition}

\begin{proof}
    Notice that $1 \leq \beta_0^{\epsilon_r}(Z) \leq \beta_0^{\epsilon_1}(Z) \leq 2T$.
    Then
    \begin{align*}
      \frac{1}{2T} \leq & \frac{\beta_0^{\epsilon_r}(Z)}{2T} \leq \frac{\beta_0^{\epsilon_1}(Z)}{2T} \leq 1 \\
      & \implies \log \left( \frac{1}{2T} \right)  \leq \log \left( \frac{\beta_0^{\epsilon_r}(Z)}{2T} \right) \leq \log \left( \frac{\beta_0^{\epsilon_1}(Z)}{2T} \right) \leq 0 \\
      & \implies - \log \left( \frac{1}{2T} \right) \geq - \log \left( \frac{\beta_0^{\epsilon_r}(Z)}{2T}  \right) \geq - \log \left( \frac{\beta_0^{\epsilon_1}(Z)}{2T}  \right) \geq 0.
    \end{align*}
    Hence if we approximate $\betaexp$ on $[\epsilon_{r_{\min}},\epsilon_{r_{\max}}]$ (i.e., for all $r \in r_{\min}, \dots, r_{\max}$), then we obtain our result. 
\end{proof}

\begin{theorem}[Positive Lyapunov Measures Yield Non-negative 0-persistence Measures] \label{thm:lambda_positive_implies_betaexp_nonnegative}
    Let $\epsilon(t) = ||\vx_A(t) - \vy_A(t)||$ for $t \geq 0 \in T$.
    Suppose 
    \[ \betaexp = - \lim\limits_{t \rightarrow \infty} \lim\limits_{\epsilon(0) \rightarrow 0} \frac{1}{\epsilon(t)} \ln \left( \frac{\beta_0^{\epsilon(t)}(\vx_A(T) \cup \vy_A(T))}{\beta_0^{\epsilon(0)}(\vx_A(T) \cup \vy_A(T))} \right), \]
    \noindent for $\beta_0^{\epsilon(t)} = \dim(H_0(\vRips(\vx_A(T) \cup \vy_A(T))))$. 
    Then $\lambda > 0 \implies \betaexp \geq 0$. 
\end{theorem}

\begin{proof}
By our assumptions, we have that 
    \begin{align*}
        \lambda > 0 & \implies \frac{1}{t} \ln \left( \frac{\epsilon(t)}{\epsilon(0)} \right) > 0 \\
        & \implies \ln \left( \frac{\epsilon(t)}{\epsilon(0)} \right) > 0 \text{ since } t > 0 \\
        & \implies \frac{\epsilon(t)}{\epsilon(0)} > 1 \\
        & \implies \epsilon(t) > \epsilon(0)\,.
    \end{align*}
    \noindent Since $\beta_0^{\epsilon(t)}$ is a monotone decreasing function of $\epsilon(t)$, the above bound implies that 
    \begin{align*}
        \beta_0^{\epsilon(t)} \leq \beta_0^{\epsilon(0)} & \implies  \frac{\beta_0^{\epsilon(t)}}{\beta_0^{\epsilon(0)}} \leq 1 \\
        & \implies \ln \left( \frac{\beta_0^{\epsilon(t)}}{\beta_0^{\epsilon(0)}} \right) \leq 0 \\
        & \implies - \ln \left( \frac{\beta_0^{\epsilon(t)}}{\beta_0^{\epsilon(0)}} \right) \geq 0 \\
        & \implies - \frac{1}{\epsilon(t)} \ln \left( \frac{\beta_0^{\epsilon(t)}}{\beta_0^{\epsilon(0)}} \right) \geq 0 \text{ since } \epsilon(t) > 0 \\
        & \implies \betaexp \geq 0.
    \end{align*}
\end{proof}

\begin{corollary}[When Positive Lyapunov Measures Imply Positive 0-persistence Measures in a Subcomplex]\label{cor:positive_lambda_implies_positive_betaexp}
    Let $\epsilon(t) = ||\vx_A(t) - \vy_A(t)||$ for $t \geq 0 \in T$.
    Define $\text{S}\vRips_\epsilon \left( \vx_A(T) \cup \vy_A(T) \right)$ to be the sub-complex of $\vRips_\epsilon \left( \vx_A(T) \cup \vy_A(T) \right)$ which only contains edges connecting points in $\vx_A(T)$ to points in $\vy_A(T)$ at the same time.
    Suppose 
    \[ \betaexp = - \lim\limits_{t \rightarrow \infty} \lim\limits_{\epsilon(0) \rightarrow 0} \frac{1}{\epsilon(t)} \ln \left( \frac{\beta_0^{\epsilon(t)}(\vx_A(T) \cup \vy_A(T))}{\beta_0^{\epsilon(0)}(\vx_A(T) \cup \vy_A(T))} \right), . \]
    \noindent for $\beta_0^{\epsilon(t)} = \dim(H_0(\text{S}\vRips(\vx_A(T) \cup \vy_A(T))))$. 
    Then $\lambda > 0 \implies \betaexp > 0$.
\end{corollary}

\begin{proof}
    By similar arguments as in the proof of  \autoref{thm:lambda_positive_implies_betaexp_nonnegative}, $\lambda > 0 \implies \epsilon(t) > \epsilon(0)$.
    By our definition of $\text{S}\vRips(\vx_A(T) \cup \vy_A(T))$, at every $\epsilon(t)$, one merge happens and hence the number of connected components decreases by one.
    This implies that 
    \[ \beta_0^{\epsilon(t)} < \beta_0^{\epsilon(0)}. \]
    \noindent Then by a similar argument as in the proof of \autoref{thm:lambda_positive_implies_betaexp_nonnegative}, this implies that $\betaexp > 0$.
\end{proof}

\begin{remark}\label{rmk:when_lambda_positive_implies_betaexp_positive}
    [When Positive Lyapunov Measures Imply Positive 0-persistence Measures in General]
    Again define the 0-persistence exponent by 
    \[ \betaexp = - \lim\limits_{t \rightarrow \infty} \lim\limits_{\epsilon(0) \rightarrow 0} \frac{1}{\epsilon(t)} \ln \left( \frac{\beta_0^{\epsilon(t)}(\vx_A(T) \cup \vy_A(T))}{\beta_0^{\epsilon(0)}(\vx_A(T) \cup \vy_A(T))} \right), \]
    \noindent for $\epsilon(t) = ||\vx_A(t)-\vy_A(t)||$ for $t \geq 0 \in T$. 
    If there exists vertices $\mathbf{z}$ $\mathbf{w}$ in $\vx_A(T) \cup \vy_A(T)$ belonging to two disconnected components of $\vRips_{\epsilon(0)}(\vx_A(T) \cup \vy_A(T)$ such that $\epsilon(0) < ||\mathbf{z} - \mathbf{w}|| \leq \epsilon(t)$, then $\beta_0^{\epsilon(t)} < \beta_0^{\epsilon(0)}$.
    Hence by the same arguments as in \autoref{thm:lambda_positive_implies_betaexp_nonnegative} and Corollary \ref{cor:positive_lambda_implies_positive_betaexp}, $\lambda > 0 \implies \betaexp > 0$.
\end{remark}

\section{An Algorithm for Computing the 0-Persistence Exponent}
\label{sec:algo}

With real data, we often only have access to a single time series.
In this case, chaos can be quantified locally within a single time series embedding.
The Kantz algorithm~\cite{KANTZ199477} estimates the Lyapunov exponent of a single trajectory in the state space embedding of a univariate time series. 
We present our pseudocode for computing the Kantz Lyapunov exponent, inspired by \cite{KANTZ199477}, in Algorithm \ref{alg:Kantz}.
We note that in \cite{KANTZ199477}, although they find stability using a similar method in systems like the H\`{e}non map, they find less stability in others like the Ikeda map.
As well, \cite{KANTZ199477} only test their methods on time series generated by models of dynamical systems, whereas we suspect that real world time series are inherently noisier and therefore may have greater effects on the Kantz estimate when subject to added noise.
Given the $i$-th observation $\vx_i \in \vR^d$ in a time series embedding with $N$ points, the average Euclidean distance between it and its neighbors is given by

\begin{equation} \label{eq:Li}
  L_i=\frac{1}{|N_\delta(\vx_i)|} \sum_{j=1}^{|N_\delta(\vx_i)|} \norm{\vx_i - \vx_j},
\end{equation}  

where $\vx_j$ is the $j$-th neighbor of $\vx_i$ in its $\delta$ neighborhood $N_\delta(\vx_i) = \{ \vx_j : \norm{\vx_i - \vx_j} < \delta, i \neq j\}$.
Moving forward $t$ steps in time, the $\delta$-neighborhood of the $(i+t)$-th point is given by $N_\delta(\vx_{i+t}) = \{ \vx_{j+t}:\vx_j \in N_\delta(\vx_i)\}$ for $t=1, \dots, t_\text{max}$. 
Similarly, the average Euclidean distance between $\vx_{i+t}$ and its neighbors is given by $\displaystyle L_{i+t}=\frac{1}{|N_\delta(\vx_{i+t})|}\sum_{j=1}^{|N_\delta(\vx_{i+t})|}\norm{\vx_{i+t}-\vx_{j+t}}$.
The scaling factor for the $i$-th observation is given by $\displaystyle \Delta S(t,i) = \frac{L_{i+t}}{L_i}$.
Fitting a linear regression to the average logarithm of scaling factors over all observations against each time step and computing the slope, we obtain an estimate of the maximal Lyapunov exponent $\lambda$ of the underlying system from which the time series data was obtained. 

\begin{algorithm}[ht!]
  \caption{Kantz algorithm: estimates the maximal Lyapunov exponent of a univariate time series}
  \label{alg:Kantz}
  \textbf{Inputs:}$X$, $\delta$, $t_{\text{max}}$, $t_{\text{initial}}$, $t_{\text{final}}$

    \smallskip
    
\hspace{0.5cm} Let $N$ be the number of observations in the embedding $X$ of a single time series. \\
\hspace{0.5cm} for $\vx_i \in X$, \\
\hspace{1cm} find all $\vx_j \in N_\delta(\vx_i)$, excluding $\vx_i$ itself \\
\hspace{1cm} compute $L_i$ (\autoref{eq:Li}). \\
\hspace{1cm} for $t \in 1, \dots, t_{\text{max}}$, \\
\hspace{1.5cm} identify all $\vx_{j+t} \in N_\delta(\vx_{i+t})$, excluding $\vx_{i+t}$ itself \\
\hspace{1.5cm} compute $L_{i+t}$ \\
\hspace{1.5cm} store $\Delta S(t,i) = \frac{L_{i+t}}{L_i}$ \\
\hspace{0.75cm} Fit a linear regression to $\ln \left( \frac{1}{N} \sum_{i=1}^N\Delta S(t,i) \right)$ against $t$ on $[t_\text{initial},t_{\text{max}}]$ and define its slope as $\lambda$. \\
   
\smallskip

\textbf{Return:} $\lambda$
\end{algorithm}

We construct a similar pipeline to that of Kantz algorithm in order to estimate the 0-persistence exponent of a real univariate time series via its state space embedding, which we call the \emph{0-persistence algorithm} (Algorithm \ref{alg:Persistence}).
Given the $i$-th observation $\vx_i$ in an embedding with $N$ points, the $0$-dimensional (scaled) Betti vector from Vietoris-Rips filtration on its $\delta$-neighborhood $\bar{N}_\delta(\vx_i) = \{\vx_j:\norm{\vx_i-\vx_j} < \delta\}$ is given by 
\begin{align*}
    \Delta \beta(\cdot,i) = \left( 1, \dots, \frac{\beta_0^{\epsilon_r}(\bar{N}_\delta(\vx_i))}{\beta_0^{\epsilon_1}(\bar{N}_\delta(\vx_i))}, \dots, \frac{\beta_0^{\epsilon_R}(\bar{N}_\delta(\vx_i))}{\beta_0^{\epsilon_1}(\bar{N}_\delta(\vx_i))} \right),
\end{align*} 
where $\displaystyle \Delta \beta(r,i) = \frac{\beta_0^{\epsilon_r}(\bar{N}_\delta(\vx_i))}{\beta_0^{\epsilon_1}(\bar{N}_\delta(\vx_i))}$ corresponds to the growth or `stretch' factor in the number of connected components in the Vietoris-Rips complex of the $\delta$ neighborhood of $\vx_i$ at the $r$-th filtration radius.
Fitting a linear regression to the negative logarithm of the average Betti number among all observations against each filtration radius and computing the slope, we obtain an estimate of the 0-persistence exponent $\betaexp$ of the underlying system from which the time series data was obtained.

\begin{algorithm}[ht!]
  \caption{0-persistence algorithm: estimates the 0-persistence exponent of a univariate time series}
\label{alg:Persistence}
\textbf{Inputs:} $X$, $\delta$, $\text{epslist} = \{ 0 < \dots < \epsilon_r < \dots < \epsilon_R \}$, $\epsilon_\text{initial}$, $\epsilon_\text{final}$

    \smallskip
    \hspace{0.5cm} Let $N$ be the number of observations in the embedding $X$ of a single time series. \\
    \hspace{0.5cm} for $\vx_i \in X$, \\
    \hspace{1cm} find all $\vx_j \in \bar{N}_\delta(\vx_i)$, \textit{including} $\vx_i$ \\
    \hspace{1cm} initialize $\text{count} = 0$ \\
    \hspace{1cm} for each $\epsilon_r \in \text{epslist}$, \\
    \hspace{1.5cm} for each $[b,d) \in \bcd_0 \Big( \vRips(\bar{N}_\delta(\vx_i)) \Big)$, \\
    \hspace{2cm} if $\epsilon_r \geq b$ and $\epsilon_r < d$, \\
    \hspace{2cm} $\text{count} = \text{count} + 1$ \\
    \hspace{1.5cm} let $\beta_0^{\epsilon_r}(\bar{N}_\delta(\vx_i)) = \text{count}$ \\
    \hspace{1.5cm} store $\Delta\beta(r,i) = \frac{\beta_0^{\epsilon_r}(\bar{N}_\delta(\vx_i))}{\beta_0^{\epsilon_1}(\bar{N}_\delta(\vx_i))}$ as the $r$th term in the Betti vector $\Delta \beta(\cdot,i)$ \\
    \hspace{0.5cm} Fit a linear regression to $-\ln\left (\frac{1}{N} \sum_{i=1}^N \Delta\beta(r,i) \right)$ against $\epsilon_r$ on $[\epsilon_\text{initial},\epsilon_\text{final}]$ and define its slope as $\betaexp$. \\
\textbf{Return:} $\betaexp$
\end{algorithm}

\begin{remark}
    [Complexity of the 0-persistence algorithm]
    \label{rmk:complexity_persistence_algo}
    In the worst case (for large $\delta$), each observation $\vx_i$ in the embedding of a time series with $N$ points has $N-1$ neighbors so that $|N_\delta(\vx_i)| = N-1$.
    Hence computing the $0$-dimensional persistent homology of $N_\delta(\vx_i)$ requires $O(N \log N)$ computations \cite{Gl2023}.
    We repeat this computation at most $N$ times, once for every observation $\vx_i$ in the embedding.
    Hence Algorithm \ref{alg:Persistence} takes $O(N^2 \log N)$ time. 
\end{remark}

\section{Results}
\label{sec:results}

We compare the stability of our 0-persistence exponent and the Lyapunov exponent when measuring chaos in two models of dynamical systems with known behavior---the Lorenz and R\"{o}ssler systems.
These systems model fluid convection and taffy pulling, respectively.
We also compare the stability of both measures computed using algorithms \ref{alg:Kantz} and \ref{alg:Persistence} on real univariate time series data of the Belousov-Zhabotinsky autocatalytic chemical reaction reported by Wodlei, Hristea, and Alberti~\cite{wodlei_2022}.
We summarize our results in the corresponding subsections.

\subsection{Lorenz System}
\label{ssec:Lorenz_results}

The Lorenz system models fluid dynamics over time. 
The differentials $dX/dt, dY/dt$, and $dZ/dt$ indicate change in fluid convection, horizontal temperature variation, and vertical temperature variation of a fluid. 
Integrating these differentials and choosing three control parameter values $\vA = \{ \sigma, \beta, \rho\}$, we define the  state space of this system as $\big( X(\vA,t), Y(\vA,t), Z(\vA,t) \big) \subset \vR^3$.
The parameters of $\vA$ describe the physical properties of the fluid, temperature difference in the fluid, and the physical properties of the overall system, respectively.
This system is mathematically modeled as follows: 
\begin{align*}
  \frac{dX}{dt}(\vA,t) & = \sigma \Big(y(t)-x(t)\Big), \\
  \frac{dY}{dt}(\vA,t) & = x(t)\Big(\rho-z(t)\Big)-y(t), \\ 
  \frac{dZ}{dt}(\vA,t) & = x(t)y(t)-\beta z(t). \\
\end{align*}
The default parameters are $\sigma = 10$ and $\beta = 8/3$, while $\rho$ is varied to analyze system dynamics.
Then the state space is given by $\Big( X(\rho,t),Y(\rho,t),Z(\rho,t) \Big)$, where a single trajectory within it is denoted as $\vx_\rho(\vT) = \{ \vx_\rho(t_0), \dots, \vx_\rho(t_{T-1}) \}$.
We choose the default initial conditions of the lorenz function in R, given by ($\vx_\rho(t_0) = \big( X(\rho,t_0), Y(\rho,t_0), Z(\rho,t_0) \big) = (-13,-14,47)$).
We fix $\vT = \{ 0, 0.01, \dots, 20 \}$.
We compute the average of each measure over 30 pairs of neighboring trajectories.
For each pair of trajectories $\{ \vx_\rho(\vT), \vy_\rho(\vT)\}$, we define the initial conditions as $\vx_\rho(t_0) = (-13,-14,47) +  \mathcal{N}(0,0.1)$ and $\vy_\rho(t_0) = \vx_\rho(t_0) + \mathcal{N}(0,0.0001)$, and the resulting trajectories by $\vx^{d,\sigma}_\rho(\vT) = \vx^d_\rho(\vT) + \mathcal{N}(0,\sigma \cdot \gamma_d)$ and $\vy^{d,\sigma}_\rho(\vT) = \vy^d_\rho(\vT) + \mathcal{N}(0,\sigma \cdot \omega_d)$, where $\gamma_d$ and $\omega_d$ are the standard deviations of the $d$-th dimension of $\vx_\rho(\vT)$ and $\vy_\rho(\vT)$, respectively, for $d=1, 2, 3$.
We vary our control parameter $\rho = \{90,90.1,\dots,103\}$ within the same range as varied in \cite{10.1063/5.0102421}, $\sigma = \{ 0,0.015,0.025 \}$, $\text{epslist} = \{ 0,0.1,\dots,8 \}$, and fit $\betaexp$ and $\lambda$ to $\epsilon = [0,6]$ and $t = [0,7]$, respectively. 

See our results in \autoref{fig:Lorenz_results}. 
As expected, our 0-persistence exponent is much more robust to small trajectory perturbations when compared to the Lyapunov exponent.
As well, our measure is also better-able to detect the periodicity around $\rho = 100.1$.
Note that by \autoref{rmk:bounds_Persistence_exp} and our $\epsilon$-fit, $\displaystyle 0 \leq \betaexp(X) \leq \frac{-\log \left( \frac{1}{2 \cdot (2000)} \right)}{6} \approx 1.38$.

\begin{figure}[ht!]
  \centering
  \includegraphics[width=\textwidth]{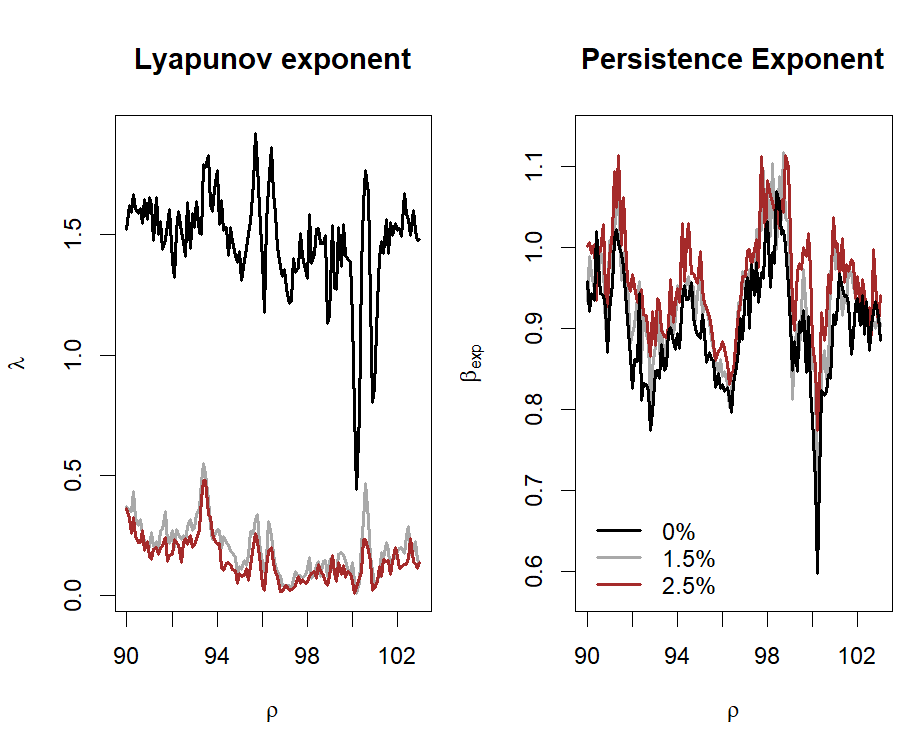}
  \caption{The average Lyapunov (left) and 0-persistence (right) exponents against increasing values of $\rho$ for varied noise levels and initial conditions in the Lorenz system.}
  \label{fig:Lorenz_results}
\end{figure}
\subsection{R\"{o}ssler System}
\label{ssec:Rossler_results}

The R\"{o}ssler system models the ribbon-like folding of taffy that occurs when a taffy pulling machine runs \cite{Rössler+1983+788+801}.
This system experiences a period-doubling route to chaos, in which its periodicity gradually increases as its control parameter varies until it mimics chaotic behavior \cite{hilborn2006}.
The differentials $dX/dt$, $dY/dt$, and $dZ/dt$ indicate changes in horizontal and vertical oscillatory behavior, and nonlinearity in the system over time.
Integrating these differentials and choosing three control parameter values $\vA = \{ a, b, c \}$ defines the state space $\Big( X,Y,Z \Big) \subset \vR^3$.
The parameters in $\vA$ control the linear stability, baseline drift, and chaoticity of the system, respectively. 
The R\"{o}ssler system is modeled as follows.
\begin{align*}
  \frac{dX}{dt}(\vA,t) & = - y(t) - z(t), \\
  \frac{dY}{dt}(\vA,t) & = x(t) + a \cdot y(t), \\
  \frac{dZ}{dt}(\vA,t) & = b + z(t) \cdot (x(t) - c).
\end{align*}

For this experiment, we follow similar parametric choices as those used by G\"{u}zel, Munch, and Khasawneh~\cite{10.1063/5.0102421}.
That is, we fix $b=2$ and $c=4$, and vary $a$ between 200 evenly-spaced values on $[0.37,0.42]$.
Then the state space is given by $\Big( X(a,t), Y(a,t), Z(a,t) \Big)$, where one of its trajectories is denoted as $\vx_a(\vT) = \{ \vx_a(t_0), \dots, \vx_a(t_{T-1}) \}$.
We fix $\vT =\{ 0, 1/15, \dots, 1000 \}$. 
We compute the average exponents over 30 pairs of trajectories.
For each pair of trajectories $\{ \vx_a(\vT), \vy_a(\vT) \}$, we define the initial conditions as $\vx_a(t_0) = (-0.4, 0.6, 1) + \mathcal{N}(0,0.001)$ and $\vy_a(t_0) = \vx_a(t_0) + \mathcal{N}(0,0.0001)$, and the resulting trajectories by $\vx_a^{d,\sigma}(\vT) = \vx_a^d(\vT) + \mathcal{N}(0, \sigma \cdot \gamma_d)$ and $\vy_a^{d,\sigma}(\vT) = \vy_a^d(\vT) + \mathcal{N}(0, \sigma \cdot \omega_d)$, where $\gamma_d$ and $\omega_d$ are the standard deviations of the $d$-th dimension of $\vx_a(\vT)$ and $\vy_a(\vT)$, respectively, for $d = 1, 2, 3$.
We set $\sigma = \{ 0, 0.0015, 0.0025 \}$, $\text{epslist} = \{ 0, 0.001, \dots, 0.5 \}$, and fit $\betaexp$ and $\lambda$ to $\epsilon = [0,0.15]$ and $t = [0,100]$, respectively.

See our results in \autoref{fig:Rossler_results}. 
We can see that $\betaexp$ is much more stable given added noise for increasing $a$ in comparison to $\lambda$.
Both measures detect periodicity in a similar manner around $a \approx 0.4$ and $a \approx 0.41$.
Note that by \autoref{rmk:bounds_Persistence_exp} and our $\epsilon$-fit, $\displaystyle 0 \leq \betaexp(X) \leq \frac{-\log \left( \frac{1}{2 \cdot (15000)} \right)}{0.15} \approx 68.73$.

\begin{figure}[ht!]
  \centering
  \includegraphics[width=0.95\textwidth]{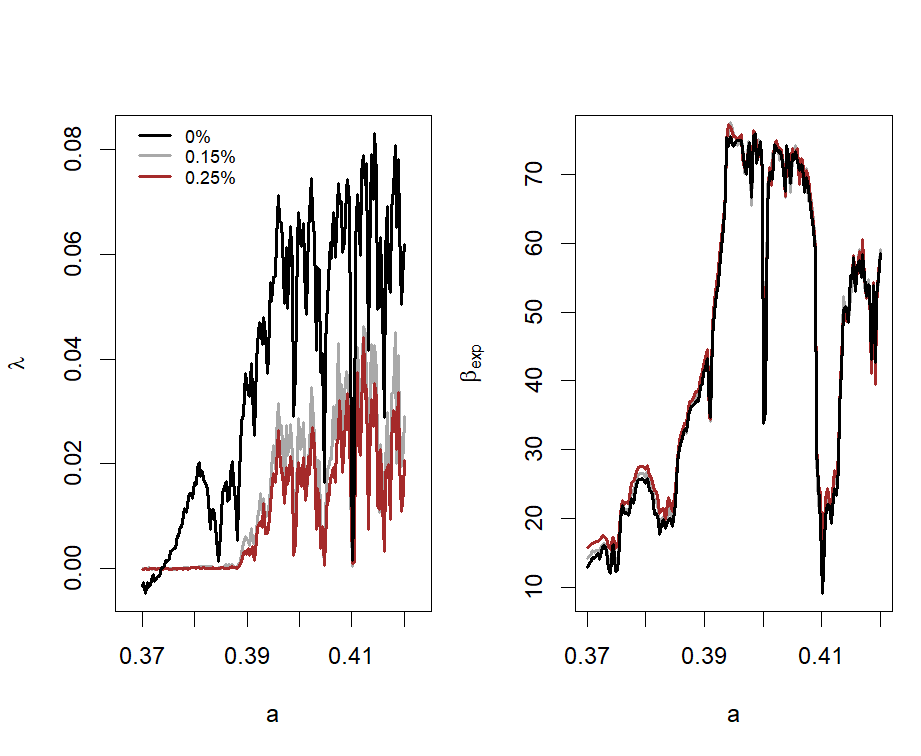}
  \caption{The average Lyapunov (left) and 0-persistence (right) exponents against increasing values of $a$ for varied noise levels and initial conditions in the R\"{o}ssler system.}
  \label{fig:Rossler_results}
\end{figure}
\subsection{Real Time Series Data}
\label{ssec:Real_results}

In this section, we compare the average stability of the 0-persistence and Lyapunov exponents of an autocatalytic chemical reaction studied by Wodlei, Hristea, and Alberti \cite{wodlei_2022}. 
We obtained a sped-up video of their experiment showing an unstirred, ferrin-catalyzed Belousov-Zhabotinsky reaction which captured the change in concentrations of the catalyst as the reaction progressed via a camera chip containing a color filter array.
We followed their experimental instructions by using ImageJ to preprocess this video data into a single time series describing the average concentration of ferrin over time.
We first extracted the average grayscale (i.e. ferrin concentration) value of each column of pixels in the video over time via a multikymograph.
To compute this graph, we extracted the average grayscale value of each column of 42 pixels centered at a yellow horizontal cross section in the video footage.
We highlight one such column of pixels at a chosen time via a white rectangle in the left image of \autoref{fig:ImageJ_results_BZrct}. 
The corresponding average grayscale values over time for this column of pixels are summarized via the white horizontal cross section in multikymograph shown in the bottom right of \autoref{fig:ImageJ_results_BZrct}.
The profile plot (i.e., time series) was then computed by averaging the grayscale values in each column of the multikymograph.
We then denoised this series via a local projection filter \cite{Kantz_Schreiber_2003}.
Since we wanted to denoise while preserving as many points as possible and any chaotic structure present, and since chaos occurs in at least three dimensions, we performed the filter with an embedding dimension of three.
We denote the resulting time series as $f_\text{avg}(\vT)$, for $\vT = \{ 0, 1, \dots,1527 \}$ the horizontal pixel values.
We define the corresponding embedding by $X = \text{SW}_{M,\tau} f_\text{avg}(\vT) = \big( f_\text{avg}(\vT),\dots,f_\text{avg}(\vT + M \tau) \big)$, where $\tau = 6$ and $M = 3$ are selected using the first minimum of the automutual information and Cao's False Nearest Neighbors test \cite{Wallot2018}.
Hence our embedding satisfies $|X| = 1510$.
We show a plot of $f_\text{avg}(\vT)$ and $X$ in the top middle and top right images of \autoref{fig:ImageJ_results_BZrct}.

To compute the maximal Lyapunov and 0-persistence exponents using Algorithms \ref{alg:Kantz} and \ref{alg:Persistence}, we first fix $\delta = c \cdot D_X$, where $D_X$ is the maximum pairwise distance between any two points in $X$.
We choose $c = 0.07$ by varying $c \in [0.01, 0.20]$ and obtaining stable measures of $\lambda \in [0.006,0.007]$ and $\betaexp \in [0.434,0.443]$ in the range $c \in [0.06,0.08]$, choosing $c$ as the mean value within this stable range.
Hence, we define $\delta \approx 23.27$ and select 100 evenly spaced filtration radii $\epsilon_r$ in $[0,\delta]$.
We fit our regression window for $\lambda$ to $t = [0, 75]$ (i.e., $t_{\max} = 75$, $t_{\text{initial}} = 0$, and $t_{\text{final}} = 75$), and for $\betaexp$ to $\epsilon = [0.2 \cdot \delta, 0.55 \cdot \delta]$ (i.e., $\epsilon_\text{initial} \approx 4.7$ and $\epsilon_\text{final} \approx 13.4$).
By \autoref{rmk:bounds_Persistence_exp}, $\displaystyle 0 \leq \betaexp(X) \leq \frac{-\log \left( \frac{1}{|X|} \right)}{8.7} \approx 0.842$.
We show the linear regressions and both exponents computed via Algorithms \ref{alg:Kantz} and \ref{alg:Persistence} in the top left and right images of \autoref{fig:Real_results}, respectively.
We can see that within this $\epsilon$-range, our embedded time-series trajectory undergoes just over half of its total possible exponential divergence.

\begin{figure}[ht!]
\centering
\includegraphics[width=\textwidth]{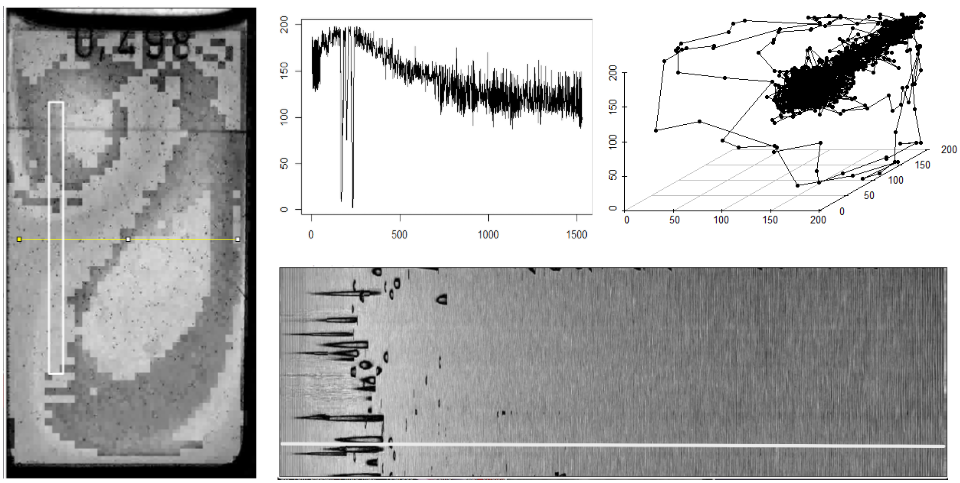}
\caption{An ImageJ screenshot of the video-d Belousov-Zhabotinski reaction~\cite{wodlei_2022} (left) and its resulting multikymograph (bottom right).
We also show the resulting denoised profile plot of average ferrin concentrations (top middle) and its sliding windows embedding (top right).}
\label{fig:ImageJ_results_BZrct}
\end{figure}

For our stability analysis, we vary $\sigma$ in $\{ 0.05, 0.1, \dots, 0.75 \}$
and define the embedding of $f_\text{avg}$ at noise level $\sigma$ as $X_\sigma = \text{SW}_{M,\tau}f_\text{avg}^\sigma(\vT)$, where $f_\text{avg}^\sigma(\vT) = f_\text{avg}(\vT) + \mathcal{N}(0,\sigma \cdot \gamma)$ is the time series at noise level $\sigma$ and $\gamma$ is the standard deviation of $f_\text{avg}(\vT)$. 
We finally compute the normalized average Lyapunov and 0-persistence exponents of the single trajectory $X_\sigma$ against $\sigma$ over 30 samples of $X_\sigma$.
For each $\sigma$, we fit our regressions for $\betaexp$ to $\epsilon \in [7,13]$ and $\lambda$ to $t \in [40,80]$.
We select our windows for $t$ and $\epsilon$ by producing a random sample of $\lambda$ and $\betaexp$ for each $\sigma$ in the list $\{ 0.1, 0.2, \dots, 0.7 \}$ and recording the window of best fit for each.
We record windows for $t$ and $\epsilon$ in the respective lists $\{ [7,13], [7,15], [7,17], [7,20] \}$ and $\{ [40,80], [30,80], [20,80], [20,90] \}$, choosing their intersections as the regression windows for each measure.
In \autoref{fig:Real_results}, can see that when adding up to 20\% Gaussian noise to $f_\text{avg}(\vT)$ (i.e., $\sigma \leq 0.2$), our 0-persistence exponent is more stable in comparison, as indicated by the normalized measures within the red vertical dashed lines laying between $[0.886,1]$ for $\lambda$ and $[0.976,1]$ for $\betaexp$.


\begin{figure}[ht!]
  \centering
  \includegraphics[width=0.95\textwidth]{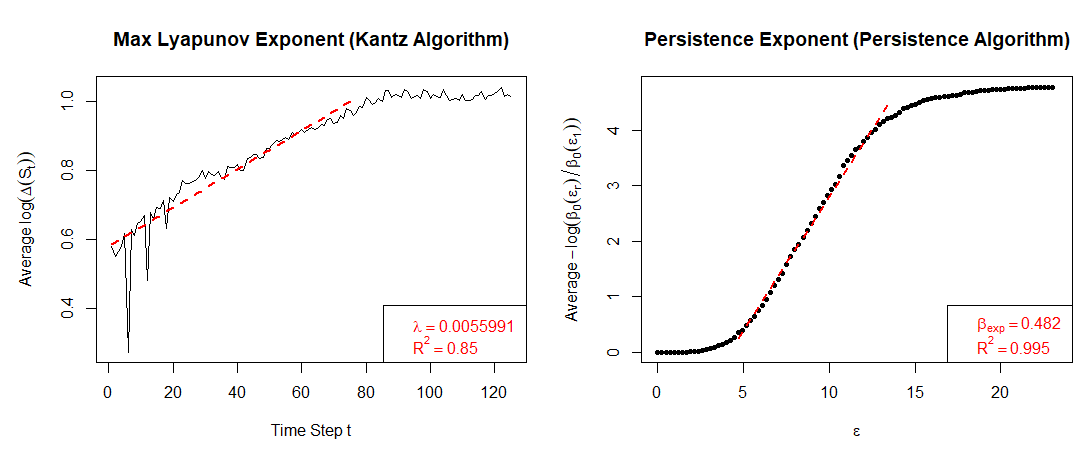}\\
  \vspace*{-0.1in}
  \includegraphics[width=0.95\textwidth]{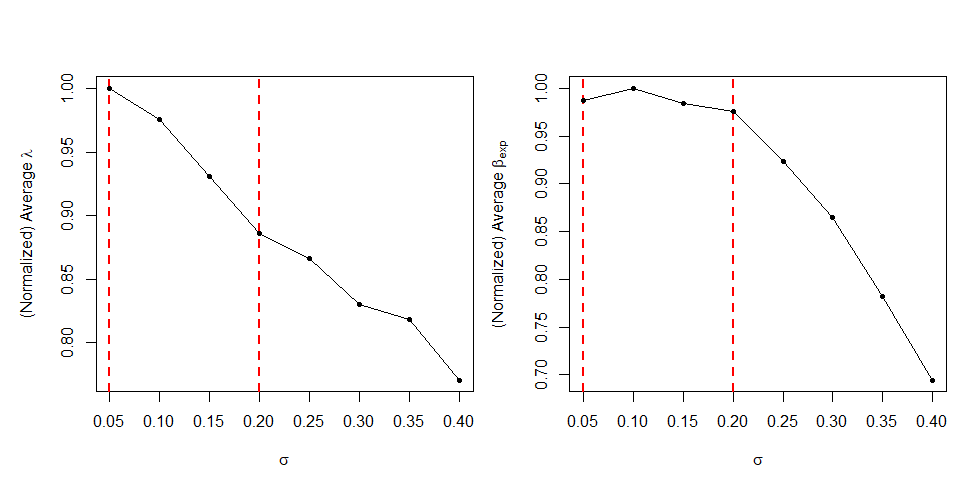}
  \caption{The maximal Lyapunov (top left) and 0-persistence (top right) exponents of the Belousov-Zhabotinsky reaction~\cite{wodlei_2022}.
  We also plot the normalized averages of the maximal Lyapunov (bottom left) and 0-persistence (bottom right) exponents against increasing Gaussian noise levels added to the time series of this reaction.
  }
  \label{fig:Real_results}
\end{figure}

\subsection{Correlation between Measures}\label{ssec:correlation}

The goal of this section is to experimentally support our hypothesis that if the maximal Lyapunov exponent identifies chaos, then so does our 0-persistence exponent (i.e.  Remark \ref{rmk:when_lambda_positive_implies_betaexp_positive}).
We prove the easy case, when we only consider the subcomplex of the Vietoris-Rips complex which only contains edges connecting points from both neighboring trajectories at the same time (\autoref{cor:positive_lambda_implies_positive_betaexp}), and it remains a goal of our future work to prove the generalized result.
For this experiment, we compare the correlation between the chaos predictions of $\lambda$ and $\betaexp$ using Matthew's Correlation Coefficient (MCC), as it is sensitive to imbalanced data \cite{ChJu2020}.
We assess the correlation on the Lorenz and R\"{o}ssler systems using the same experimental set-ups as defined in the respective subsections \ref{ssec:Lorenz_results} and \ref{ssec:Rossler_results}.

For each system, we define our chaos predictions as follows.
For measures produced from Lorenz trajectories, if $\lambda > c_{\lambda} \cdot 1.8$ and $\betaexp > c_{\betaexp} \cdot 1.1$ where $c_{\lambda}, c_{\betaexp} = \{0.2,0.3,\dots,0.8\}$, then $\lambda$ and $\betaexp$ predict the system to be chaotic (i.e. the true chaos as measured by $\lambda$ is denoted as the binary value of $1$ while the predicted chaos as measured by $\betaexp$ is denoted by $1$).
Otherwise, $\lambda$ predicts no chaos (i.e. the `true' chaos is captured by a 0 and the predicted chaos is captured by a 1).
We choose our values of $1.8$ and $1.1$ by the estimated upper bounds shown in Figure \ref{fig:Lorenz_results}.
We repeat this process for the R\"{o}ssler measures, but instead consider true and predicted measures of chaos to occur when $\lambda > c_{\lambda} \cdot 0.085$ and $\betaexp > c_{\betaexp} \cdot 75$.

For every combination $\{ c_{\lambda},c_{\betaexp} \}$, we compute the MCC score for predicted and true chaos values from $\betaexp$ and $\lambda$ over increasing values of $\rho$ (Lorenz) and $a$ (R\"{o}ssler).
We compute the predicted and true response for each value of $\rho$ and $a$ by using the average measures of $\betaexp$ and $\lambda$ generated over 20 pairs of neighboring trajectories sampled from the corresponding system.
We again define neighboring trajectories and $\rho$,$a$ values using the same set-up as described in subsections \ref{ssec:Lorenz_results} and \ref{ssec:Rossler_results}.
For each observed MCC, we approximate its significance via permutation testing \cite{NICHOLS2007253} by randomly permuting the predicted chaos responses from $\betaexp$ 500 times and recording the resulting MCC for each permutation.
We then define the $p$-value for the observed MCC as the proportion of random MCCs at least as large as the observed coefficient.
See \autoref{fig:correlation_results} for our results.
We show a heatmap of MCCs produce for each combination of $c_{\lambda}$ and $c_{\betaexp}$.
We overlay the $p$-value for each MCC in white text.
From these findings, we can see that significant correlation between both $\lambda$ and $\betaexp$ (i.e. MCC $> 0.5$, $p < 0.05$) is identified for both Lorenz and R\"{o}ssler systems.

\break

\begin{figure}[ht!]
\centering
\includegraphics[width=0.45\textwidth]{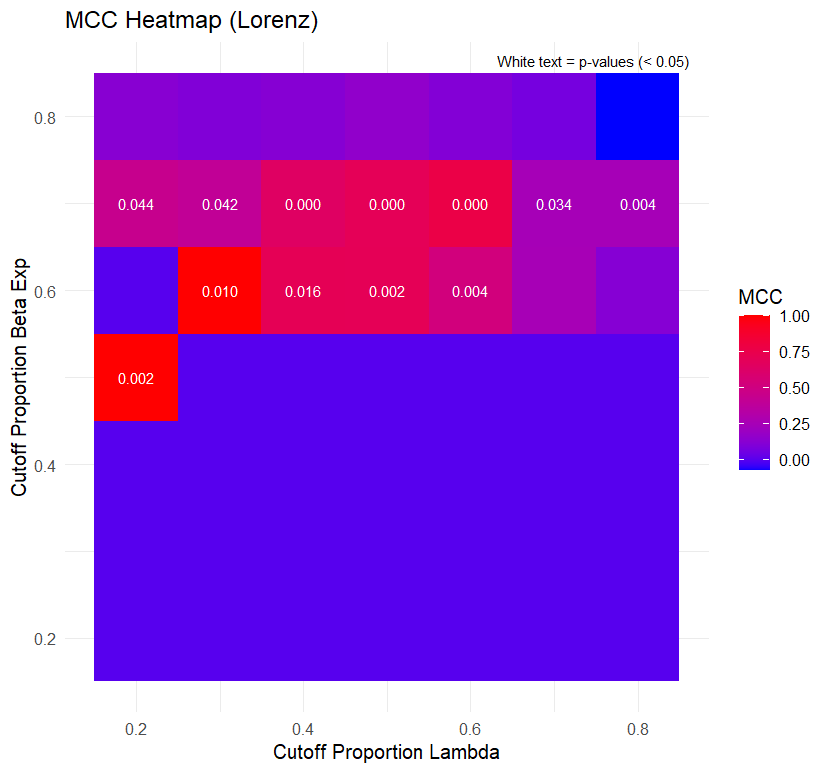}
\includegraphics[width=0.45\textwidth]{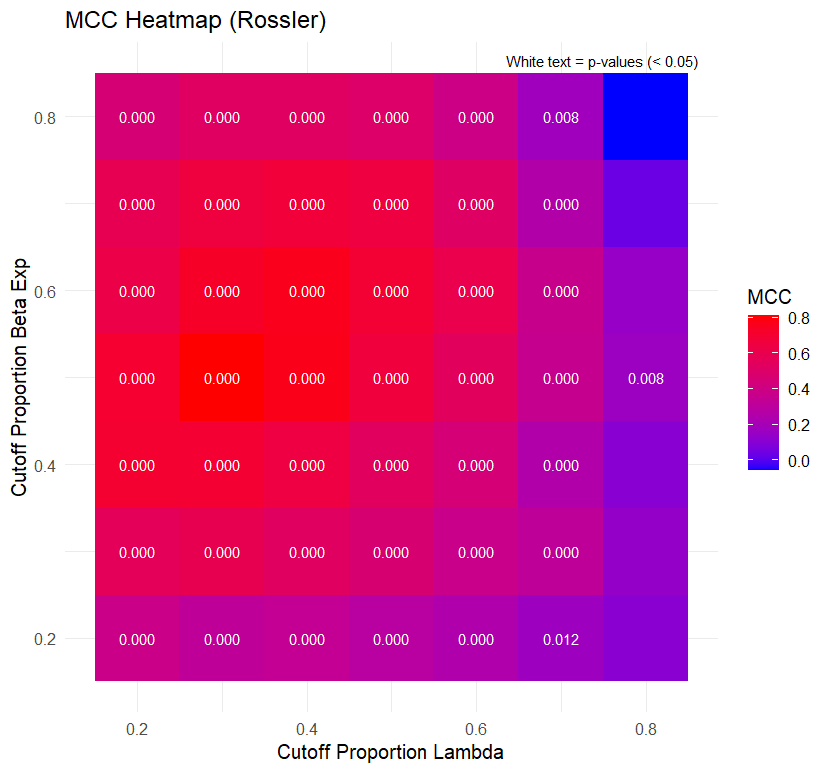}
\caption{Heatmap of Matthew's correlation coefficients and $p$-values for binary cutoffs used to produce chaos predictions with $\betaexp$ and true responses with $\lambda$.
We show results for the Lorenz (left) and R\"{o}ssler (right) systems.}
\label{fig:correlation_results}
\end{figure}

\section{Discussion}
\label{sec:discussion_future}

Although our work successfully provides a stable alternative to the Lyapunov exponent for noisy trajectories of dynamical systems, our 0-persistence exponent is only an approximation of chaos since it is only defined for a pair of trajectories sampled from a system, whereas the Lyapunov exponent is a measure defined for the entire system.
Despite this, our measure does show promising theoretical and experimental advantages when measuring chaos from real-world dynamics as indicated by our univariate time series results. 
Not only this, but our measure lays the groundwork for future intersections of dynamical systems and quantifying chaos with persistent homology.

We are further interested in exploring extensions of our exponent to higher homology dimensions.
However, in any dimension greater than zero, small trajectory perturbations can cause small higher-dimensional holes to appear suddenly for a short time during filtration, hence decreasing the stability of the corresponding Betti curve.
As such, we want to explore the use of a stable Betti number \cite{ScChLuRaOb2017} to construct such a chaos measure. 
Persistent homology has also been used to construct a measure of quasiperiodicity for systems \cite{TrPe2018}.
We have shown our measure to be robust to added noise in systems displaying period-doubling routes to chaos (i.e., the R\"{o}ssler system), however we are interested in using this quasiperiodicity measure to construct a measure of chaos robust to added noise for systems that undergo a quasiperiodic route to chaos \cite{hilborn2006}.
We further want to explore algorithm optimization for Algorithm \ref{alg:Persistence} as well as analyzing the behavior of $\betaexp$ for varying control parameters in higher dimensions (i.e., when we vary more than one parameter at a time to detect dynamic behavior).
Furthermore, for systems whose dynamics are better described by more than one parameter, we are interested in constructing a measure for chaos via multiparameter persistent homology.

\bibliography{timeseries,homology,DimRedn,Refs_Prot,dynamics,correlation}

\end{document}